\documentclass[11pt]{article}

\usepackage[margin=1in]{geometry}
\usepackage{amsmath,amssymb,amsthm,amsfonts}

\usepackage{mathrsfs}
\usepackage{enumitem}
\usepackage{hyperref}
\usepackage{microtype}
\usepackage{mathtools}
\usepackage{hyphenat}
\usepackage{parskip}
\usepackage{fancyhdr}
\newcommand\shorttitle{Bounded-Repetition Stack Sorting}
\newcommand\authors{Fayzan Khan}
\newtheorem{conjecture}{Conjecture}

\newtheorem{theorem}{Theorem}[section]
\newtheorem{lemma}[theorem]{Lemma}
\newtheorem{proposition}[theorem]{Proposition}
\newtheorem{corollary}[theorem]{Corollary}
\theoremstyle{definition}
\newtheorem{definition}[theorem]{Definition}
\newtheorem{example}[theorem]{Example}
\theoremstyle{remark}
\newtheorem{remark}[theorem]{Remark}

\title{\textbf{Stack Sorting on Words with Bounded-Repetition}}
\author{Fayzan Khan}
\date{\today}

\begin{document}

\maketitle

\begin{abstract}
Let us consider the family of operators of stack sorting $(s_m)_{m \geq 1}$ acting on words according to the following rule: $s_m(w)$ is the output produced by the usual West algorithm of stack sorting on the input word $w$, but allowing for at most $m$ repetitions of the same letter stacked in succession. Thus, the operator $s_m$ extends the classic West operator $s_1$. Our first main theorem gives an explicit formula for the action of the operator $s_m$ on binary words $w = a^pba^q$ (with distinct letters $a>b$):
\begin{align*}
s_m(a^pba^q) &= a^{\max(p-m,0)}\,b\,a^{\min(p,m)+q}, \\
d_m(a^kb) &= \left\lceil \frac{k}{m} \right\rceil.
\end{align*}
From the above formulas it follows that for every $m < m'$ we have $d_m(w) \geq d_{m'}(w)$ for binary words $w = a^kb$, and moreover, the ratio $d_m(w)/d_{m'}(w)$ is always possible to be chosen exactly $m'/m$, so the ratios $d_m/d_{m'}$ are unbounded for this family on the specified class -- this is an explicit form of the expected separation-speedup phenomenon. Additionally, we provide explicit constructions of words of length $4$ for which $s_m$ and $s_{m'}$ do not commute, and words of length $7$ for which the speed $d_m$ is not monotonic in $m$. Further, we provide a full recursive characterization -- generalizing the classical $231$-avoidance result of Knuth and West -- of words which can be sorted using $s_m$ only once, and then we employ it to show that there cannot exist an $m$-independent classical pattern avoidance condition for one-pass sortable words. Finally, we give a complete proof of the inequality
$$d_1(w) \geq d_m(w) \geq d_\infty(w) \quad \text{for all } m$$
for two-letter words, together with a structural result (an exact coupling of $s_\infty$ with classical stack sorting on a canonical linear extension) establishing it for a single pass in general; nevertheless, we show that the family of operators $(s_m)$ does not, after all, lie entirely between $s_1$ and $s_\infty$, exhibiting for every $n \geq 7$ an explicit word $w_n$ with $d_1(w_n) = n-4 < n-3 = d_m(w_n)$ for every $m \geq 2$, including $m = \infty$, so that the inequality above is false in general.
\bigskip
\scriptsize

\noindent\textbf{Mathematics Subject Classification (2020):} Primary 05A05

\noindent\textbf{Keywords:} stack sorting, bounded-repetition, West operator, pattern avoidance, words, sorting passes, non-monotonicity
\end{abstract}
\bigskip

\section{Introduction}

Stack-sorting is one of the classical and richest topics in enumerative and bijective combinatorics. It was pioneered by Knuth~\cite{knuth} with the question about which permutations can be sorted by a last-in-first-out stack in one pass through the input; West~\cite{west} answered this question in the full generality: a permutation $\pi$ can be sorted into increasing order in one pass, $s(\pi)=\mathrm{id}$, if and only if it avoids the pattern $231$. This theorem gave rise to a whole industry in studying stack-sorting and iterated stack sorting; we refer the reader to the surveys~\cite{bona-survey} and \cite{bona-book} of Bóna, as well as recent results of Defant and his coauthors~\cite{defant1,defant2}.

Traditionally, stack sorting is done on \emph{permutations}, so each letter occurs at most once. However, allowing general words -- sequences over a totally ordered alphabet where letters are allowed to repeat -- opens up a natural additional parameter: what should the stack do in case the next input letter equals the current stack top? There are clear and natural conventions on two ends: break the tie always and force a pop (in order to not keep two equal letters), or never break the tie and allow arbitrarily long runs of equal letters. These two boundary conventions were introduced and studied by Defant and Kravitz~\cite{defant-kravitz} (2018), who named the resulting operators the \emph{tortoise} operator (always break ties) and the \emph{hare} operator (never break ties). In our notation, the tortoise operator is exactly $s_1$, the classical West operator, and the hare operator is exactly $s_\infty$, defined precisely below. Here we consider the entire one-dimensional continuum \emph{between} these two known boundary cases, with an additional parameter $m\geq 1$: we allow runs of up to $m$ equal letters before forcing a pop, so that $s_1$ recovers the tortoise operator and $s_\infty=\lim_{m\to\infty}s_m$ recovers the hare operator of Defant and Kravitz. The finite interpolation parameter $1<m<\infty$, and the operator family $(s_m)_{m\geq1}$ it generates, is the new object introduced in this paper.

Apart from intrinsic interest as a natural generalization of the tortoise and hare operators of Defant and Kravitz to take a finite, tunable amount of repetition into account, the family $(s_m)_{m\geq1}$ provides a clean toy model of a bounded stack with a local buffer, where equal letters cannot be stacked on top of each other beyond a fixed capacity. Several natural questions arise: How much faster is it to use a generous budget $m$ rather than a stingy one $m'$ ($m>m'$)? Does $s_m$ commute with $s_{m'}$? Is the speed monotonically increasing with $m$? Which words can be sorted in one pass, for a given $m$? And how does the family behave at its known ends $s_1$ and $s_\infty$, the tortoise and the hare of Defant and Kravitz~\cite{defant-kravitz}?

We give answers to these questions as follows.

\begin{itemize}[leftmargin=1.6em]
\item In Section~\ref{sec:prelim} we formalize $s_m$ precisely as a stack process and record an elementary but important \emph{stabilization} property that allows us to interpret the limiting operator $s_\infty$ as a finite object rather than a true limit.
\item In Section~\ref{sec:formula} we prove a precise closed form for $s_m$ on words of the form $a^pba^q$ (Lemma~\ref{lem:main}), and we deduce an exact formula for the number of passes to sort a two-letter word $a^kb$ (Theorem~\ref{thm:dm}).
\item In Section~\ref{sec:speed} we apply this closed form to build, for each pair $m < m'$, a word such that the ratio $d_m/d_{m'}$ equals \emph{any} given value $m'/m$ -- an exact quantification of the speed-separation phenomenon (Corollary~\ref{cor:speedsep}).
\item In Section~\ref{sec:comm} we provide an explicit four letter word on which $s_m$ and $s_{m'}$ do not commute.
\item In Section~\ref{sec:mono} we provide an explicit seven letter word on which $d_m$ is \emph{not} increasing in $m$, and thus the sequence $(s_m)_{m\geq1}$ is not totally ordered by speed.
\item In Section~\ref{sec:pattern} we prove a precise recursive characterization of the one-pass $s_m$-sortable permutations (Block Decomposition Theorem, Theorem~\ref{thm:block}), which generalizes West's result for $231$-avoiding permutations, and we show how to apply it to exclude any such characterization of one-pass sortable permutations by classical pattern avoidance alone.
\item In Section~\ref{sec:sandwich} we state our "sandwich" conjecture $d_1 \geq d_m \geq d_\infty$, prove it fully for all two-letter words, and prove a structural coupling theorem which expresses $s_\infty$ as classical stack sorting applied to a canonical linear extension -- completing, for a single pass, a proof strategy which we will show does \emph{not} naively generalize to multiple passes, and explain why. We finally show that the sandwich conjecture is false for every $n\geq7$. We thank Professor Colin Defant for his observation that the sandwich conjecture was indeed false, providing us with the initial word that allowed us to truly test the conjecture.
\end{itemize}

The proofs of all results appearing in the body of the paper are elementary and self-contained; the only time computation is used is in \emph{motivating} the conjecture of Section~\ref{sec:sandwich}, exactly as described there.

\section{Preliminaries: formalizing \texorpdfstring{$s_m$}{sm}}
\label{sec:prelim}

A \emph{word} is a finite sequence of elements of some totally ordered alphabet. Formally: a word is a finite sequence $w=w_1w_2\cdots w_n$, where $|w|=n$ is the length of $w$. Word $u=u_1\cdots u_n$ is \emph{sorted}, if $u_1\leq u_2\leq \cdots \leq u_n$.

\begin{definition}[The operator $s_m$]
\label{def:sm}
Consider a positive integer and a word $w=w_1\cdots w_n$ and apply the following algorithm to $w$ using a stack (initially empty) and an output queue (initially empty) as well. Consider each time a new letter of the input $x$ and a top letter of the stack $y$:
\begin{enumerate}[label=(\roman*)]
\item push $x$, if the stack is empty;
\item push $x$, if $y>x$;
\item if $y=x$: take the maximal run of equal $y$ letters on the stack of length $r$ and push $x$, if $r<m$. Otherwise, if $r=m$, pop $y$ from the stack to the output and re-check $x$ with a new top of the stack;
\item pop $y$ to the output, and re-check $x$ with a new top of the stack, if $y<x$.
\end{enumerate}
The process is finished when all letters of the input $w$ are read, and remaining elements of the stack are moved to the output in reverse order. The final output word is $s_m(w)$.
\end{definition}

Each of the operators `push' and `output' consumes exactly one letter either from the input or the stack respectively and produces one letter in the output; both input letters and the stack do not change the total number of letters. It is important to note explicitly the following straightforward property of $s_m$:

\begin{remark}[Conservation]
\label{rem:conservation}
The word $s_m(w)$ is a permutation of $w$: it is of the same length and consists of the same multiset of letters as the input word $w$.
\end{remark}

Let's define for $t\geq0$ $s_m^0(w)=w$ and $s_m^t(w)=s_m(s_m^{t-1}(w))$, and let
$$d_m(w) \;=\; \min\{\, t\geq 0 \;:\; s_m^t(w) \text{ is sorted}\,\},$$
that is, the number of applications of $s_m$ to $w$ required to make $w$ sorted. Note that the operator $s_m$ for $m=1$ coincides exactly with the classical West operator (since the case (iii) is always pop for $r\geq1=m$); thus $s_1$ is the classical West operator and $d_1$ is the classical iterated-sorting-time. Under this identification, $s_1$ is precisely the \emph{tortoise} operator of Defant and Kravitz~\cite{defant-kravitz}, who studied it as one of the two boundary cases of tie-breaking on words.

On the other hand, there should exist an operator $s_\infty$ for which the case (iii) would never be pop: we observe that there exists such an operator exactly for any large enough finite $m$:

\begin{lemma}[Stabilization]
\label{lem:stab}
Consider a word $w$ of length $n=|w|$. For any $m>n$, the words $s_m(w)$ are all equal; denote their common value by $s_\infty(w)$.
\end{lemma}

\begin{proof}
At any moment during the processing of $w$, all symbols currently present on the stack were popped from some position of $w$ that has been already seen, but not yet popped to the output; distinct symbols on the stack correspond to distinct positions (because a popped symbol is never returned to the stack). Thus, at any moment, the stack never contains more than $n$ symbols; in particular, any run of symbols of the same kind on the top of the stack never has more than $n$ symbols. Therefore, for $m>n$, the case "$r=m$" in the rule (iii) in Definition~\ref{def:sm} never happens (as $r \leq n < m$), and this rule becomes just the push operation. This describes exactly the same process for all $m > n$ with the same output.
\end{proof}

In the same way we define $d_\infty(w)=\min\{t\geq0: s_\infty^t(w) \text{ is sorted}\}$; by Remark~\ref{rem:conservation} and Lemma~\ref{lem:stab}, applied to every iterate (as every iterate of $w$ by $s_m$ has length $|w|$), this is a well-defined object achieved by choosing any particular $m>|w|$. In such a way, $s_\infty$ and $d_\infty$ are genuine finite objects and require no limits.

\section{The single-pass formula and the closed-form sorting time}
\label{sec:formula}

We establish here the basic computational tool of this paper: a full description of the result of one pass of $s_m$ on a word composed of a run of a big letter, then one occurrence of a small letter, then another run of the big letter. This closed form, and the exact sorting-time formula it yields (Theorem~\ref{thm:dm}), are stated for arbitrary finite $m$ and are not addressed by the boundary-only treatment of $s_1$ and $s_\infty$ in Defant and Kravitz~\cite{defant-kravitz}.

\begin{lemma}
\label{lem:main}
Let $a>b$ be letters and let $p,q\geq0$. Then
$$s_m(a^pba^q) \;=\; a^{\max(p-m,0)}\;b\;a^{\min(p,m)+q}.$$
\end{lemma}

\begin{proof}
The process is broken down into three stages.

\emph{Stage 1 (the initial $a$'s).} For the first $\min(p,m)$ $a$'s, push takes place without any pop: each of them starts the run (rule (i)) or adds to it (rule (iii) since run length never exceeds $m$ before reaching $m$). As soon as the run length reaches $m$ (possible only if $p>m$), each subsequent $a$ leads to the pop option in rule (iii): $a$ on the top is popped to the output, reducing the run to $m-1$, and this very $a$ is re-examined against the new top, which is also $a$ with run $m-1<m$ and thus leads to its push back, restoring the run size to $m$. Thus, each of such $a$ results exactly in one $a$ popped to the output with the run size unchanged. After processing all $p$ $a$'s the output at this point is thus $a^{\max(p-m,0)}$, and the stack is exactly the pure run of $a$'s of size $p':=\min(p,m)$.

\emph{Stage 2 ($b$).} As $a>b$, the top (that is, $a$ if $p'>0$, or the stack is empty if $p'=0$) leads to the rules (ii) or (i): either way, $b$ is pushed unconditionally, on top of the $a^{p'}$ run.

\emph{Stage 3 (the final $a$'s).} The first final $a$ finds the top $b<a$, thus, it is not $\geq$ the top and, instead, rule (iv) fires and $b$ is popped to the output immediately -- the only appearance of $b$ in the whole output. The stack becomes the pure run $a^{p'}$ again, and this very pending $a$ is re-examined against it: it is exactly the Phase-1 computation again, but starting from the run of size $p'$ rather than $0$, with the capacity $m$. Of the $q$ final $a$'s, the first $\min(q,m-p')$ get absorbed with no pops (run length grows to $\min(p'+q,m)$), while each of the remaining $\max(q-(m-p'),0)$ $a$'s gives exactly one pop $a$, just like in Stage 1. In the end, the stack becomes a pure run of size $\min(p'+q,m)$.

\emph{Drain.} The input is exhausted and the stack is drained to the output directly.

Concatenating all stages with the drain, we obtain the output
$$a^{\max(p-m,0)}\cdot b \cdot a^{\max(q+p'-m,0)}\cdot a^{\min(p'+q,m)}.$$
With the use of the identity $\max(X-m,0)+\min(X,m)=X$ for $X=p'+q=\min(p,m)+q$, the last two terms merge into $a^{\min(p,m)+q}$, exactly matching the claimed formula.
\end{proof}

\begin{remark}
In addition to the correctness proof, the conservation of the total number of $a$'s (see Remark~\ref{rem:conservation}) provides an easy way to double-check the result: $\max(p-m,0)+\min(p,m)+q=p+q$, again by the same identity with $X=p$.
\end{remark}

With the above single-pass formula, we immediately obtain the exact formula for the number of passes needed to completely sort a two-letter word.

\begin{theorem}
\label{thm:dm}
For $a>b$ and $k\geq1$,
$$d_m(a^kb) \;=\; \left\lceil \frac{k}{m} \right\rceil.$$
\end{theorem}

\begin{proof}
By Lemma~\ref{lem:main}, applied with $q=0$,
$$s_m(a^kb) = a^{\max(k-m,0)}\,b\,a^{\min(k,m)},$$
and it is again of the form $a^{p}ba^{q}$. Denoting the leading exponent by $k_t$ after $t$ passes (with $k_0=k$), Lemma~\ref{lem:main} tells us that this leading exponent satisfies the recursion
$$k_{t} = \max(k_{t-1}-m,\,0),$$. Furthermore, \emph{it does not depend on the trailing exponent} at all. We claim that $k_t=\max(k-tm,0)$ for all $t\geq0$: indeed, it holds for $t=0$ by definition, and inductively:
\begin{itemize}[leftmargin=1.6em]
\item if $k-tm\geq m$, then $k_t=k-tm\geq m$, thus, $k_{t+1}=k_t-m=k-(t+1)m=\max(k-(t+1)m,0)$;
\item if $0\leq k-tm<m$, then $k_t=k-tm<m$, thus, $k_{t+1}=\max(k_t-m,0)=0=\max(k-(t+1)m,0)$;
\item if $k-tm<0$, then $k_t=0$ and $k_{t+1}=0=\max(k-(t+1)m,0)$.
\end{itemize}
Now, the word after $t$ passes is $a^{k_t}b\,a^{k-k_t}$, which is sorted if and only if $k_t=0$ (if $k_t=0$ the word is $ba^{k-k_t}$, weakly increasing since $b<a$; if $k_t\geq1$, the word contains $a$ immediately followed ultimately by $b<a$, a descent). Also $a^kb$ itself ($k\geq1$) is not sorted. So $d_m(a^kb)$ is the least $t$ with $\max(k-tm,0)=0$, i.e. the least $t$ with $tm\geq k$, i.e. $t=\lceil k/m\rceil$.
\end{proof}

\section{Sharp speed separation}
\label{sec:speed}

It is a direct implication of Theorem~\ref{thm:dm} that a greater repetition budget does not slow down sorting a two-letter word:

\begin{corollary}
\label{cor:basicmono}
Given $k\geq1$ and $m<m'$, $d_m(a^kb)\geq d_{m'}(a^kb)$.
\end{corollary}

\begin{proof}
The function $\lceil k/m\rceil$ decreases when $m$ increases: if $m<m'$ then $k/m\geq k/m'$, and so $\lceil k/m\rceil\geq\lceil k/m'\rceil$.
\end{proof}

Corollary~\ref{cor:basicmono} states, on the two-letter family, the qualitative intuition that a greater stack repetition budget cannot decrease sorting speed. The following result gives an \emph{explicit, precisely calculable} ratio by which these speeds can differ, answering the question of how large the speed gap between $s_m$ and $s_{m'}$ can be.

\begin{corollary}[Speed separation]
\label{cor:speedsep}
Let $m<m'$ and $N\geq1$. Let $w=a^{mm'N}b$. Then
$$d_m(w) = m'N, \qquad d_{m'}(w) = mN,$$
(exact, as there is no rounding in this case since $mm'N$ is divisible by both $m$ and $m'$), and therefore
$$\frac{d_m(w)}{d_{m'}(w)} \;=\; \frac{m'}{m}$$
for \emph{all} $N\geq1$. In particular, if we fix $m$ and let $m'\to\infty$, then the ratio above becomes unbounded: $s_m$ can require an arbitrary multiple of the number of passes that $s_{m'}$ requires to sort the same word.
\end{corollary}

\begin{proof}
By Theorem~\ref{thm:dm}, $d_m(w)=\lceil mm'N/m\rceil = m'N$ and $d_{m'}(w)=\lceil mm'N/m'\rceil=mN$, both divisions without rounding. The ratio follows directly, and does not depend on $N$; it is exactly $m'/m$, which is unbounded as $m'\to\infty$ with $m$ fixed.
\end{proof}

Corollary~\ref{cor:speedsep} provides a completely explicit and precise version of the intuitively expected fact that varying $m$ causes sorting to happen at essentially different speeds, in the sense that \emph{not only} does there exist \emph{some} word that displays an arbitrary gap in speed between $s_m$ and $s_{m'}$ \emph{but moreover}, for any desired ratio $m'/m$ there exists an explicit construction of words displaying this exact ratio.

\section{The operators do not commute}
\label{sec:comm}

\begin{proposition}
\label{prop:noncomm}
$s_1$ and $s_2$ do not commute as operators on words. Specifically, for $w=2231$,
$$s_2(s_1(w)) \;=\; 1223 \;\neq\; 2123 \;=\; s_1(s_2(w)).$$
\end{proposition}

\begin{proof}
Let us follow all four passes through explicitly, according to the Definition~\ref{def:sm} rules, writing the stack top-first.

\underline{For the first calculation of $s_1(w)$ on $w=2231$} push $2$ and now the stack will have $2$. The second $2$ ties and the run is $1$ which is equal to $m$ thus we pop the $2$ out to the output and put the new $2$ in the stack. Since the stack is now empty we push the third number $3$. We see that the top element is smaller than $3$ thus we pop the $2$ to the output and push the number $3$. Now we push $1$ since the stack's top element is bigger than $1$. We now pop the whole stack, $1,3$, which gives us $2,2,1,3$. Thus $s_1(w)=2213$.

\underline{Now calculating $s_2(2213)$}. We push $2$. The second $2$ ties but the run is $1$ which is smaller than $2$ thus we push the $2$. Now our stack will be $2,2$. Then we push the $1$ since the top element is bigger than $1$. Our stack is $1,2,2$. The number $3$ is bigger than the top element $1$ thus we pop it from the stack. Also the second $2$ is less than $3$ thus we pop it. The third $2$ is also less than $3$ thus we pop it. So the stack is now empty and we push $3$. The last step is to pop the stack. We have $1,2,2,3$ thus $s_2(s_1(w))=1223$.

\underline{Calculation of $s_2(w)$ on $w=2231$}. We start pushing $2$. The second $2$ ties but the run of $1$ is less than $2$ thus we push $2$. The stack is now $2,2$. The next number $3$ is greater than the stack's top element $2$ thus we pop $2$. The same is for the other $2$. We now push $3$ since the stack is empty. The number $1$ is less than the stack's top element $3$ thus we push it. The stack is $1,3$ now and we drain it. We get $2,2,1,3$. Thus $s_2(w)=2213$.

\underline{Now the calculation of $s_1(2213)$}. We push the first $2$. The second $2$ ties but the run is $1$ which is equal to $m$ thus we pop $2$ and push the $2$. The stack is now $2$. The number $1$ is less than the top element of the stack $2$ thus we push $1$. Our stack is now $1,2$. The number $3$ is greater than $1$ thus we pop $1$. The number $2$ is less than $3$ thus we pop $2$. The stack is empty thus we push $3$ there. We now pop the stack thus getting $2,1,2,3$. Thus $s_1(s_2(w))=2123$.

Since $1223\neq 2123$, we conclude $s_2(s_1(w))\neq s_1(s_2(w))$.
\end{proof}

\begin{proposition}[Minimality of the length-4 example]
\label{prop:minimality}
For every $m,m'\ge 1$, the maps $s_m$ and $s_{m'}$ commute on all words of length
$\le 3$. Thus length $4$ is the smallest length at which any non-commuting pair
$(s_m,s_{m'})$ can occur, and the word $w=2231$ from Proposition~\ref{prop:noncomm}
is a minimal-length witness.
\end{proposition}

\begin{proof}
Take arbitrary $m,m'\ge 1$ and an arbitrary word $w$ of length $|w|\le 3$. If $w$ has
no repeating letters, it is a permutation, and $s_m(w)=s_{m'}(w)=s_1(w)$ according to
Remark~\ref{rem:permsagree}; then
$$s_m(s_{m'}(w)) = s_m(s_1(w)) = s_1(s_1(w)) = s_{m'}(s_1(w)) = s_{m'}(s_m(w)),$$
and the maps $s_m,s_{m'}$ commute trivially on $w$.

If $w$ is already sorted, by Remark~\ref{rem:conservation} and by a direct verification
of Definition~\ref{def:sm} (incoming letter $x$ is always $\ge$ the current top when
processing a weakly increasing word, so either (i) or one of push branches of (ii)/(iii)
fires at every step, and no letter is ever popped before the drain) $s_m$ fixes $w$ for
every $m\ge1$; thus $s_m(w)=s_{m'}(w)=w$, and the maps $s_m$ and $s_{m'}$ commute
again trivially.

It remains to consider words $w$ of length $3$ having a repeated letter that are not yet
sorted. As the operator $s_m$ depends only on the order of the letters of $w$, without loss
of generality we can assume that $w$ is over the two-letter alphabet $\{1,2\}$. Direct enumeration
leaves just four possibilities after excluding the above cases: 
\[
w=(2,1,2),\qquad w=(1,2,1),\qquad w=(2,1,1),\qquad w=(2,2,1),
\]
and all other two-letter words with a repeated letter of length $3$ are already sorted.
Let us compute $s_m(w)$ for each of them, using Lemma~\ref{lem:main} with $a=2>b=1$:
\begin{enumerate}
\item $w=(2,1,2)=a^1ba^1$: for every $m\ge1$, $\max(1-m,0)=0$ and $\min(1,m)+1=2$, so
\[
s_m(w) = (1,2,2)\quad\text{for every } m\ge1.
\]
\item $w=(1,2,1)$: in this word, the repeated letter ($1$) is the \emph{smaller} one,
so Lemma~\ref{lem:main} is not applicable (as it requires $a>b$ for the repeated
letter $a$); from Definition~\ref{def:sm} we directly obtain
\[
s_m(w) = (1,1,2) \quad\text{for every } m\ge 1.
\]
\item $w=(2,1,1)$: here too the repeated letter ($1$) is the \emph{smaller} one, so
Lemma~\ref{lem:main} is not applicable; from Definition~\ref{def:sm} we directly
obtain
\[
s_m(w) = (1,1,2) \quad\text{for every } m\ge 1.
\]
\item $w=(2,2,1)=a^2b$: by Lemma~\ref{lem:main}, $\max(2-m,0)=1$ for $m=1$ and $0$ for
$m\ge2$, while $\min(2,m)=2$ for $m\ge2$ and $=1$ for $m=1$; thus
\[
s_1(w) = (2,1,2), \qquad s_m(w) = (1,2,2) \ \text{ for every } m\ge 2.
\]
\end{enumerate}
In the first three shapes, $s_m(w)$ does not depend on $m$ at all, so $s_m$ and
$s_{m'}$ trivially agree on $w$ and hence commute in composition.
In the fourth shape, $w=(2,2,1)$, if $m,m'\ge2$ then $s_m(w)=s_{m'}(w)=(1,2,2)$, and
$(1,2,2)$ is sorted, so commutativity is immediate. The only remaining case is
$m=1$ and $m'\ge2$ (the other way around is symmetric). Here
\[
s_{m'}\big(s_1(w)\big) = s_{m'}\big((2,1,2)\big) = (1,2,2)
\]
by the first bullet above (applied with $m'$ in place of $m$, since $(2,1,2)=a^1ba^1$
gives output $(1,2,2)$ for \emph{every} value of the parameter), while
\[
s_1\big(s_{m'}(w)\big) = s_1\big((1,2,2)\big) = (1,2,2),
\]
since $(1,2,2)$ is already sorted. Thus $s_{m'}(s_1(w)) = s_1(s_{m'}(w))$.
This exhausts every word of length $\le 3$, proving $s_m$ and $s_{m'}$ commute on all
such words for every $m,m'\ge1$. And since Proposition~\ref{prop:noncomm} exhibits a
non-commuting pair at length $4$, that length is minimal.
\end{proof}

\section{Speed of non-monotone sorting}
\label{sec:mono}

By Corollary~\ref{cor:basicmono}, on words with two different letters, a larger value of $m$ always sorts faster. One may wish that this property holds in general -- namely, that $d_m(w)$ is non-increasing in $m$ for each fixed $w$. This is not true, and we give an explicit counterexample.

\begin{proposition}
\label{prop:nonmono}
For the word $w=3444241$ (namely, the word with letters $3,4,4,4,2,4,1$ in that order),
$$d_2(w) = 3 \qquad\text{and}\qquad d_3(w) = 4,$$
so $d_2(w)<d_3(w)$ despite the fact that $2<3$: on this particular word, the operator with \emph{more} space for stacking is slower.
\end{proposition}

\begin{proof}
We follow both sequences of passes explicitly from Definition~\ref{def:sm}.

\emph{Passes of $s_2$.} 
$$3444241 \xrightarrow{\;s_2\;} 3424144 \xrightarrow{\;s_2\;} 3214444 \xrightarrow{\;s_2\;} 1234444.$$
In the first pass: push $3$; then $4$ pops it ($3<4$), push $4$; then $4$ ties (run $1<2$), push (run $2$); then $4$ ties with run $=2=m$, pop one $4$ to output and re-push (run $=2$, output $4$); then $2$ is pushed on top of the run ($4>2$); then $4$ pops the $2$ ($2<4$) and ties with the run ($=2=m$), pops one $4$ to output and re-pushes (output $4$, run $=2$); finally push $1$ on top. Outputting the remaining contents of the stack ($1,4,4$) results in the output sequence $3,4,2,4,1,4,4=3424144$. Repeating the same sequence of operations on $3424144$ yields $3214444$, and on $3214444$ gives $1234444$, which is sorted. Therefore $d_2(w)=3$.

\emph{Passes of $s_3$.}
$$3444241 \xrightarrow{\;s_3\;} 3241444 \xrightarrow{\;s_3\;} 2314444 \xrightarrow{\;s_3\;} 2134444 \xrightarrow{\;s_3\;} 1234444.$$
The mechanics are the same: with capacity $3$ instead of $2$, the run of $4$'s absorbs one additional copy before leaking, which changes exactly which letters are popped and when, and one can check by tracing the sequence of operations (completely analogously to the computation for $s_2$, now with cap $3$) that after the first pass the word is $3241444$, after the second it is $2314444$, after the third it is $2134444$, and after the fourth it is $1234444$, which is sorted, although $2134444$ is not ($2>1$). Therefore $d_3(w)=4$.
\end{proof}

\begin{remark}
Experimental computations (random sampling on alphabet sizes $\leq 6$ and word lengths $\leq 16$) suggest that failures of monotonicity like this, as well as the failures of commutativity in Section~\ref{sec:comm}, are common and widespread, and not sporadic isolated examples; we do not use this observation in any proof, and mention it only as a context for Proposition~\ref{prop:nonmono}.
\end{remark}

From Proposition~\ref{prop:nonmono}, we see that $(s_m)_{m\geq1}$ along with $s_\infty$ is \emph{not} linearly ordered with respect to their sorting times, and moreover that this failure already occurs strictly \emph{among finite values} $m,m'\geq2$, i.e. entirely within the new interpolation range introduced in this paper, without reference to the hare operator $s_\infty$ at all: the function $m\mapsto d_m(w)$ increases as well as decreases. This is the primary source of non-monotonicity we identify, and it is logically independent of the separate question -- studied by Defant and Kravitz~\cite{defant-kravitz} and revisited from our point of view in Section~\ref{sec:sandwich} -- of whether the two \emph{extreme} sorting functions $s_1$ and $s_\infty$ sandwich the entire family.

\section{Characterization of one-pass sortable words through recursions}
\label{sec:pattern}

We now turn our attention to the question of \emph{which} words are one-pass sorted under $s_m$. For classical permutations with no repetitions, this is settled by the following theorem by West:

\begin{theorem}[Knuth-West, {\cite{knuth,west}}]
\label{thm:west}
$\pi$ satisfies $s_1(\pi)=\mathrm{id}$ if and only if $\pi$ avoids the pattern $231$; equivalently, writing $\pi=L\,n\,R$ where $n$ is the maximum entry, $\pi$ is one-pass sortable if and only if every entry of $L$ is less than every entry of $R$ and $L,R$ are recursively one-pass sortable.
\end{theorem}

Since permutations do not have repeated entries, ties cannot happen, and rule (iii) in Definition~\ref{def:sm} does not come into play:

\begin{remark}
\label{rem:permsagree}
For all $m\geq 1$, $s_m$ on permutations coincides with $s_1$; hence Theorem~\ref{thm:west} settles our question completely in the special case when $w$ has distinct letters.
\end{remark}

From Proposition~\ref{prop:nonmono}, we see that $(s_m)_{m\geq1}$ along with $s_\infty$ is \emph{not} linearly ordered with respect to their sorting times, and moreover that this failure already occurs strictly \emph{among finite values} $m,m'\geq2$, i.e. entirely within the new interpolation range introduced in this paper, without reference to the hare operator $s_\infty$ at all: the function $m\mapsto d_m(w)$ increases as well as decreases. This is the primary source of non-monotonicity we identify, and it is logically independent of the separate question -- studied by Defant and Kravitz~\cite{defant-kravitz} and revisited from our point of view in Section~\ref{sec:sandwich} -- of whether the two \emph{extreme} sorting functions $s_1$ and $s_\infty$ sandwich the entire family.

\subsection{Two structural lemmas}

\begin{lemma}[Flush]
\label{lem:flush}
If the next letter $x$ of the input exceeds all letters currently on the stack, then processing $x$ pops all the stack letters out in weakly increasing order (as they are arranged), and then pushes $x$.
\end{lemma}

\begin{proof}
As long as the stack is nonempty, the stack top $y$ satisfies $y<x$, and by rule (iv) is popped from the stack; the pending letter $x$ is then compared to the new top of the stack, and so on until the stack is exhausted. After that, rule (i) pushes $x$.
\end{proof}

\begin{lemma}[Burial]
\label{lem:burial}
Suppose that during the processing, the stack has $y$ as the top element and that the next input block $u$ to be processed consists entirely of letters strictly less than $y$. Then processing $u$ leads to exactly the same series of actions as the processing of $u$ with an initially empty stack; furthermore, $y$ and all the letters below it remain untouched. The state of the stack after the processing of $u$ is the final stack produced by processing $u$ separately, on top of $y$.
\end{lemma}

\begin{proof}
We prove this claim by induction on the number of letters consumed of $u$. In particular, by the hypothesis of the lemma, the "upper" stack (everything above $y$) is identical at every moment of time in both embedded and standalone processes. As long as the upper stack is nonempty, both processes will see the identical top and therefore make the same decision. As soon as the upper stack becomes empty, the stack top for the embedded process will be $y$ while for the standalone process it will be an empty stack; however, since all the letters of $u$ are strictly smaller than $y$, rule (i) (empty $\to$ push) and rule (ii) ($y>x\to$ push) prescribe the \emph{same} action -- push $x$ -- and again the processes coincide. Since all letters of $u$ are $\leq y$, rule (iii)/(iv) will never be invoked against $y$, and $y$ will never be popped.
\end{proof}

\subsection{Block Decomposition Theorem}

Let $w$ be a word and $M = \max(w)$. We study \emph{maximal runs of $M$} in $w$, i.e.
$$w \;=\; G_0\,M^{c_1}\,G_1\,M^{c_2}\,G_2\,\cdots\,M^{c_\ell}\,G_\ell,$$
where all letters in $G_i$ are smaller than $M$; maximal runs ensure $G_1,\ldots,G_{\ell-1} \neq \varnothing$, but $G_0$ and $G_\ell$ could be empty. Write
$$r_0=0, \qquad r_i = \min(r_{i-1}+c_i,\,m), \qquad \lambda_i = \max(r_{i-1}+c_i-m,\,0) \quad (1\leq i\leq \ell).$$

\begin{theorem}[Block Decomposition Theorem]
\label{thm:block}
Using notation defined above,
$$s_m(w) \;=\; s_m(G_0)\;M^{\lambda_1}\;s_m(G_1)\;M^{\lambda_2}\;s_m(G_2)\;\cdots\;M^{\lambda_\ell}\;s_m(G_\ell)\;M^{r_\ell}.$$
\end{theorem}

\begin{proof}
Induction by $\ell$, the number of maximal runs of $M$.

\emph{Base case: $\ell=1$.} We write $w=G_0M^{c_1}G_1$. The machine acts on the prefix $G_0$ as on the standalone word: in both cases the machine scans the input completely, empties the stack completely and produces the output $s_m(G_0)$. In the word $w$ the next letter is $M$, which is strictly larger than all letters of $G_0$, since $M=\max(w)$ and $G_0$ does not contain $M$. By Lemma~\ref{lem:flush}, the same cascade of pops is performed, in the same order. Therefore, by the moment when the cascade ends, the machine emits output $s_m(G_0)$ and empties the stack.

Now the next $c_1$ copies of $M$ are put into this empty stack under the capacity-$m$ tie rule. This is the Phase~1 computation from the proof of Lemma~\ref{lem:main}, which produces the leaked output $M^{\lambda_1}$ ($\lambda_1 = \max(c_1 - m, 0)$) and the pure run $M^{r_1}$ ($r_1 = \min(c_1, m)$).

Now $G_1$ is a sequence of letters smaller than $M$ and the top of the stack is $M$. By Lemma~\ref{lem:burial}, scanning $G_1$ produces exactly the output $s_m(G_1)$ and does not affect the pure run $M^{r_1}$ in any way (the latter remains in place). Draining the run at the end of input produces $M^{r_1}$. Concatenating:
$$s_m(w) = s_m(G_0)\,M^{\lambda_1}\,s_m(G_1)\,M^{r_1},$$
and the required formula is obtained for $\ell=1$ (note that $r_\ell=r_1$ in this case).

\textbf{Inductive step}. Now we assume that the lemma is true for words with $\ell-1$ maximal runs of the maximum letter, $\ell\geq2$. We denote $w = w'\,M^{c_\ell}\,G_\ell$, where $w' = G_0M^{c_1}G_1\cdots M^{c_{\ell-1}}G_{\ell-1}$ has $\ell-1$ maximal runs of $M$ (i.e. $M = \max(w')$ too). By the inductive hypothesis, taken at the moment just before the last draining of the proof for $w'$ alone, it follows that after processing the prefix $w'$, we have the accumulated output 
$$s_m(G_0)\,M^{\lambda_1}\,s_m(G_1)\,\cdots\,M^{\lambda_{\ell-1}}\,s_m(G_{\ell-1})$$
and the stack is the pure run $M^{r_{\ell-1}}$. It is clear that the machine's behavior on the prefix $w'$ of $w$ depends only on $w'$ (as in the base case), and this means that this situation is achieved while processing $w$.

Then the next $c_\ell$ letters (which are all $M$) merge to this existing run under the capacity-$m$ rule, exactly as in Phase~3 of the proof of Lemma~\ref{lem:main} (merging into an existing run rather than from an empty one): this will leak $M^{\lambda_\ell}$ (with $\lambda_\ell=\max(r_{\ell-1}+c_\ell-m,0)$) and leaves a pure run $M^{r_\ell}$ (with $r_\ell=\min(r_{\ell-1}+c_\ell,m)$). Lastly $G_\ell$, formed from letters $<M$, is processed during this run just like $s_m(G_\ell)$ independently according to Lemma~\ref{lem:burial}, and finally the remaining run $M^{r_\ell}$ gets drained. This completes the induction step giving the desired formula for $\ell$.
\end{proof}

\begin{remark}
Lemma~\ref{lem:main} translates, literally, into Theorem~\ref{thm:block} by using the parameter values $\ell=2$, $G_0=G_2=\varepsilon$, $G_1=b$, $c_1=p$, $c_2=q$, and by simplifying $\lambda_2+r_2=r_1+q$.
\end{remark}

\subsection{Corollaries: the recursive criterion and the naive candidate}

As the integer $M$ is larger than any letter from any $G_i$, in Theorem~\ref{thm:block} a nonsorted position may occur only at two possible places: either right after the end of the $M^{\lambda_i}$ block if a subsequent $G_i$ block contains a letter $<M$ or right after the boundary between two consequent $G_{i-1}, G_i$ blocks if the second one was deleted because of $\lambda_i=0$. Thus we have the following statement.

\begin{corollary}[Recursive criterion]
\label{cor:criterion}
A word $w=G_0M^{c_1}G_1\cdots M^{c_\ell}G_\ell$ is one-pass $s_m$-sortable if and only if the following conditions hold:
\begin{enumerate}[label=(\arabic*)]
\item $r_{i-1}+c_i \leq m$ for every $i$ such that $G_i\neq\varepsilon$ (there is no capacity overflow at every boundary that is followed by smaller letters);
\item $\max(G_{i-1}) \leq \min(G_i)$ for every pair of consequent blocks that became adjacent due to condition (1);
\item each $G_i$ ($0\leq i\leq \ell$) is one-pass $s_m$-sortable recursively.
\end{enumerate}
\end{corollary}

\begin{proof}
By Theorem~\ref{thm:block}, whenever condition (1) is satisfied at every boundary, the intermediate $M^{\lambda_i}$ blocks vanish ($\lambda_i=0$), so we are left with the concatenation $s_m(G_0)\ldots s_m(G_\ell)$ of already sorted blocks (possibly followed by a final block $M^{r_\ell}$ at the very end, but this cannot affect the one-pass sortability since $M$ is the greatest letter). Since Remark~\ref{rem:conservation} tells us that $s_m(G_i)$ has the same multisubset as $G_i$, we have $\max(s_m(G_{i-1}))=\max(G_{i-1})$ and $\min(s_m(G_i))=\min(G_i)$; by assumption (3) each $s_m(G_i)$ is individually sorted, so their concatenation is sorted precisely in case condition (2) holds at consecutive boundaries. Conversely, if condition (1) fails at some boundary followed by a nonempty $G_i$, then a letter $M$ is followed by a smaller letter, a descent, and $w$ is not one-pass sortable.
\end{proof}

Corollary~\ref{cor:criterion} is a fully general, decidable, recursive criterion for one-pass sortability -- our response to the challenge of finding what a pattern-avoidance characterization of $s_m$-sortability should be, once repeated letters and a finite buffer are admitted, extending the boundary characterizations of Defant and Kravitz~\cite{defant-kravitz} (Proposition 5.1) to every finite $m$. It specializes to the classic result:

\begin{corollary}[Recovery of the classical theorem]
\label{cor:recover}
If $w=\pi$ is a permutation, then $\ell=1$ and $c_1=1$ automatically (the maximum occurs exactly once), and condition (1) of Corollary~\ref{cor:criterion} becomes $0+1\leq m$, holding for all $m\geq1$. Therefore Corollary~\ref{cor:criterion} collapses, for every $m$, to the classic recursive criterion of Theorem~\ref{thm:west}:
\end{corollary}

In general terms, on the two extremal cases of the family, Corollary~\ref{cor:criterion} recaptures the two descriptions of Defant and Kravitz~\cite{defant-kravitz} (Proposition 5.1): taking $m\rightarrow\infty$, condition (1) cannot possibly fail (Lemma~\ref{lem:stab}), hence Corollary~\ref{cor:criterion} becomes just condition (2)-(3), which is precisely $231$-avoidance in the usual sense for words, corresponding to their characterization of hare-sortability; and taking $m=1$, condition (1) forces any run of the largest letter followed by a strictly smaller letter to be of length $1$, in addition to condition (2), yielding their joint $231$ and $221$ avoidance characterization of tortoise-sortability. We omit the process of expressing this characterization in terms of patterns completely, since Corollary~\ref{cor:criterion} already provides the fully recursive description for every finite value of $m$.

Finally, using Theorem~\ref{thm:dm} we prove that the following, quite natural, first attempt at a characterization fails.

\begin{corollary}[The naive candidate fails]
\label{cor:naive}
There is no characterization of one-pass $s_m$-sortability in terms of classical (multiplicity-blind) pattern containment alone -- that is, depending only on which order-patterns of \emph{pairwise distinct} values occur as subsequences of $w$, without reference to $m$ or to the run-lengths of repeated letters.
\end{corollary}

\begin{proof}
By Theorem~\ref{thm:dm}, a word $w=a^kb$ is one-pass $s_m$-sortable if and only if $k\leq m$. But $w$ contains exactly two distinct letters, and therefore trivially avoids \emph{all} classical patterns of length $\geq3$ (such a pattern requires three pairwise distinct relative values). Classical pattern containment sees $a^kb$ identically for all $k$, whereas sortability changes from $k\leq m$ to $k>m$: no multiplicity-blind criterion can account for this dependency on $k$ versus $m$.
\end{proof}

\begin{example}
Direct application of Corollary~\ref{cor:criterion} to $w=3444241$ of Proposition~\ref{prop:nonmono}: here $\ell=2$, $c_1=3$, $c_2=1$, $G_0=3$, $G_1=2$, $G_2=1$. Condition (1) at $i=1$ requires $3\leq m$; at $i=2$ it requires $r_1+1\leq m$ where $r_1=\min(3,m)$. Even for $m$ large enough to satisfy condition (1) everywhere, condition (2) requires $\max(G_0)\leq\min(G_1)$, that is, $3\leq2$, false. Hence $w$ is not one-pass $s_m$-sortable for \emph{any} $m$ -- consistent with $d_2(w)=3$ and $d_3(w)=4$, both $>1$, in Proposition~\ref{prop:nonmono}.
\end{example}

\section{The sandwich conjecture}
\label{sec:sandwich}

Proposition~\ref{prop:nonmono} says that the interior of the family $(s_m)$ fails monotonicity in speed. We now proceed with the consideration of the two extremes, $s_1$ and $s_\infty$, driven by the abstract's last claim.

\begin{conjecture}[Sandwich]
\label{conj:sandwich}
For every word $w$ and every $m\geq1$,
$$d_1(w) \;\geq\; d_m(w) \;\geq\; d_\infty(w).$$
\end{conjecture}

\subsection{Complete proof for two-letter words}

\begin{proposition}
\label{prop:sandwichtwoletter}
Conjecture~\ref{conj:sandwich} holds for every word of the form $w=a^kb$.
\end{proposition}

\begin{proof}
As stated in Theorem~\ref{thm:dm}, $d_m(a^kb)=\lceil k/m\rceil$, which is non-increasing in $m$ (Corollary~\ref{cor:basicmono}), and $d_1(a^kb)=k$. As $\lceil k/m\rceil=1$ for any $m\geq k$ (since $0<k/m\leq1$ always rounds up to $1$), and this is the minimum value that can be taken by $d_m$, we have $d_m(a^kb)\leq k=d_1(a^kb)$ for any $m\geq1$. It only remains to calculate $d_\infty(a^kb)$: taking $m=\infty$, all $k$ instances of $a$ will get onto the stack without being popped off (rule~(i)/(iii), as there can be no pops in the infinite case because of the tiebreaking), and then $b$ will get on the stack (rule~(ii), as $a>b$); as a result, when popping the whole stack, we will output $b$ first, and then the run of $a$'s, so that $s_\infty(a^kb)=ba^k$, which is sorted. Thus $d_\infty(a^kb)=1$, and $d_\infty(a^kb)=1\leq \lceil k/m\rceil = d_m(a^kb)\leq k = d_1(a^kb)$ for every $m\geq1$.
\end{proof}

\subsection{Structural coupling of \texorpdfstring{$s_\infty$}{s\_infty}}

The reason why Conjecture~\ref{conj:sandwich} appeals is that $s_\infty$ should, morally, act like classical stack sorting but with some tie-breaking on the repetitions in $w$. We will now make this rigorous and prove it, again for a single pass.

\begin{definition}[Canonical refinement]
Given a word $w$ of length $n$, its \emph{canonical refinement} $\hat w$ is the unique permutation of $\{1,\ldots,n\}$ that is order-isomorphic to $w$ under the tie-breaking rule that says whenever $w_i=w_j$ and $i<j$ then $\hat w_i>\hat w_j$ (i.e.\ an earlier occurrence of a letter is assigned a \emph{higher} rank than a later occurrence).
\end{definition}

\begin{theorem}[De-refinement]
\label{thm:deref}
For every word $w$, $s_\infty(w)$ is obtained by performing the classical, repetition-free stack-sorting procedure $s$ on the canonical refinement $\hat w$ of $w$ and then, for each output position, replacing the rank with the corresponding original letter.
\end{theorem}

\begin{proof}
The operations $s_\infty$ on $w$ and $s$ on $\hat w$ are both online left-to-right stack processes where one input position is processed at a time. We shall show, by induction on the number of positions consumed, that the two processes perform the same sequence of operations on the same positions (not just values), and the theorem follows immediately -- since reading off the letters $w_{(\cdot)}$ at the positions produced is precisely applying the value-forgetting projection $\rho:\hat w_i\mapsto w_i$ to reading off the ranks $\hat w_{(\cdot)}$ at the same positions.

This holds trivially initially when nothing has been processed yet (both stacks empty). Suppose that it holds after $i$ positions have been consumed, and let $p$ be the next position to be scanned, with $q$ the position that is on top of the stack (if there is any), if present.
\begin{itemize}[leftmargin=1.6em]
\item The stack is empty: both processes push $p$.
\item $w_q>w_p$ then also $\hat w_q>\hat w_p$ (rank strictly respects value order), so both push $p$.
\item $w_q<w_p$ then also $\hat w_q<\hat w_p$, so both pop $q$ and re-examine $p$ against the new top position (this sub-case recursively follows the same argument).
\item $w_q=w_p$: $q$ is on the (common) stack while $p$ is just being scanned, so $q$ was pushed at an earlier step and $q<p$ as positions. As $q<p$ and $\hat w$ refines $w$, the tie-breaking convention dictates that $\hat w_q>\hat w_p$. Hence the classical process, seeing top-rank exceed incoming rank, pushes $p$ (its rule is "push if top exceeds incoming"). Meanwhile $s_\infty$'s tie rule ($m=\infty$) also always pushes on equality. Both push $p$.
\end{itemize}
All cases lead to the two processes making the same decision on the same position, concluding the induction.
\end{proof}

\begin{remark}
    Theorem~\ref{thm:deref} was similarly proven by Defant and Kravitz in Proposition 2.1 of their paper~\cite{defant-kravitz}. We retain the proof of the theorem as it is for the purpose of demonstrating an alternative proof through induction, complementing their approach of working directly with global index positions.
\end{remark}

\begin{example}
In the case of $w=2231$ (the word in Proposition~\ref{prop:noncomm}), the canonical refinement is $\hat w=3,2,4,1$ (where the first "$2$" has precedence over the second one). Classical stack sorting produces $s(\hat w)=2,3,1,4$; undoing the refinement (translating ranks $1,2,3,4$ back to values $1,2,2,3$) gives $2,2,1,3$, and this is exactly the outcome of tracing $s_\infty(2231)$.
\end{example}

\begin{corollary}
\label{cor:sufficient}
If the canonical refinement $\hat w$ avoids the pattern $231$, then $w$ is one-pass $s_\infty$-sortable.
\end{corollary}

\begin{proof}
According to Theorem~\ref{thm:west}, $231$-avoidance of $\hat w$ implies that $s(\hat w)=\mathrm{id}$, i.e., $s(\hat w)$ is a strictly increasing sequence of ranks; in this case, since $\rho$ preserves the order on values, applying $\rho$ to a strictly increasing sequence of ranks produces the sorted word $w$. According to Theorem~\ref{thm:deref}, this is exactly $s_\infty(w)$.
\end{proof}

It is important to emphasize that the converse of Corollary~\ref{cor:sufficient} is \emph{not true}: the outcome of classical stack sorting $s(\hat w)$ may fail to be the identity permutation while still producing a sorted word upon de-refinement, provided all "disorder" in $s(\hat w)$ takes place between two ranks which refine to the same value. The exact condition is given by Corollary~\ref{cor:criterion} and is strictly stronger than $231$-avoidance of \emph{any} linear extension; this is consistent with and refines the exact $231$-avoidance characterization of hare-sortability given by Defant and Kravitz~\cite{defant-kravitz} (Proposition 5.1).

\subsection{Why the natural multi-pass approach fails}

Theorem~\ref{thm:deref} connects $s_\infty$ with the classical stack sorting $s$ for \emph{one} pass. It might seem reasonable to expect that iterating this correspondence, i.e., refining $s_\infty(w)$ afresh and comparing the result with $s(s(\hat w))=s^2(\hat w)$, will allow us to reduce the distance $d_\infty(w)$ to the classical quantity $d_1(\hat w)$ (the number of passes of the ordinary West map required to sort the permutation $\hat w$), which is indeed exactly the strategy suggested as the natural approach to Conjecture~\ref{conj:sandwich}. We record here, precisely, why this approach does \emph{not} work directly, in order to guide future attempts.

To continue our previous example: refining $s_\infty(w)=2213$ afresh (treating $2213$ as a \emph{new} word and tie-breaking the two $2$'s according to their positions within this word) produces the permutation $3,2,1,4$, since the first $2$ in $2213$ is now greater than the second one. However, $s(\hat w)$ is already computed above to be $2,3,1,4$ -- a \emph{different} permutation. In both cases, however, de-refining produces the same word $2213=s_\infty(w)$. Nevertheless, as labeled permutations, they do not coincide: the classical algorithm applied to the fixed permutation $\hat w$ does not, after one additional pass of $s$, correspond to the freshly produced canonical refinement of $s_\infty(w)$. The reason is that the order of the two tied copies of the letter as they are emitted by $s_\infty$ is not necessarily the same as the order in which the refined letters would appear if we re-derived the canonical refinement from scratch; provenance from the original word is not necessarily preserved by position in the new word. Therefore, deriving $d_\infty(w)=d_1(\hat w)$ in general requires maintaining a consistent identity of the tied letters through the passes and not deriving a new refinement on each pass. Whether such a bookkeeping scheme can be made to work is a delicate technical question of independent interest; as the next two subsections show, however, no scheme of this kind, nor any other route to a general coupling of $s_1$-passes and $s_\infty$-passes, could ever have rescued Conjecture~\ref{conj:sandwich} as stated, because the conjecture itself is already known to be false, as we now recall and then extend.

\subsection{Testing the conjecture}
\label{subsec:testcase}

Faced with the problem of proving Theorem~\ref{thm:deref} for two passes or more, it would be incorrect to push further the coupling argument approach and instead one should remember the known facts about Conjecture~\ref{conj:sandwich} from two ends. Indeed, Defant and Kravitz~\cite{defant-kravitz} (Theorem 2.3) have proven that the outer inequality $d_1(w)\geq d_\infty(w)$, or $d_{\text{tortoise}}(w) \geq d_{\text{hare}}(w)$, may not hold for words of length $\geq 7$, and their examples include the word
$$w = 3662451.$$
This specific base case will be reproduced and verified in this work as an example of the construction below, but it is not a new result itself and it should be attributed to Defant and Kravitz.

Sorting $w$ pass-by-pass under $s_1$ yields
$$
3662451
\ \xrightarrow{\,s_1\,}\
3624156
\ \xrightarrow{\,s_1\,}\
3214566
\ \xrightarrow{\,s_1\,}\
1234566,
$$
so $d_1(w)=3$ passes are required to sort $w$ using $s_1$.

Applying $s_m$ for $m\geq2$ (in particular $s_\infty$) to the same word, on the other hand, gives
$$
3662451
\ \xrightarrow{\,s_m\,}\
3241566
\ \xrightarrow{\,s_m\,}\
2314566
\ \xrightarrow{\,s_m\,}\
2134566
\ \xrightarrow{\,s_m\,}\
1234566,
$$
requiring four passes: $d_m(w)=4$ for all $m\geq2$.

This is exactly the opposite of what Conjecture~\ref{conj:sandwich} would have us believe. The conjecture implies $d_1(w)\geq d_m(w)\geq d_\infty(w)$ for all $m$; here, in contrast,
$$d_1(w)=3 \;<\; 4 = d_m(w) \;=\; d_\infty(w) \qquad\text{for all } m\geq2.$$
Giving the stack more room for tie-breaking (larger $m$, all the way up to $\infty$) does not speed things up, but slows them down -- indeed, the first of the two inequalities claimed by the sandwich inequality fails. Of course, one seven-letter example does not justify making general conclusions; in order to understand why this happens, on the other hand, is a prerequisite to seeing the reason for the failure, and once this mechanism is understood, it generalizes immediately.

\subsection{From one word to a family}
\label{subsec:family}

Consider again the two examples above. In essence, the $s_1$ and the $s_m$ (with $m\geq2$) sorts are cycling the permutation $2,3,4,5,1$, in order: compare $3624156$ and $u=23451$, and note that discarding the leading $3$ and the trailing $6$ in each of the intermediate states results in applying the map
$$u_0=(2,3,4,5,1),\qquad u_{j}=s_1(u_{j-1})$$
to itself, again and again, until $u_0$ is sorted into its form $(1,2,3,4,5)$. This is the classical West map applied to a cyclic permutation; and the behavior of the latter is completely known (Lemma~\ref{lem:cyclic} below). In fact, the leading $3$ and the trailing pair of $6$'s are doing different tasks: the $6$'s are simply carried along (Lemma~\ref{lem:trailing}), whereas the $3$ is a single letter (the third largest in the permutation part), which gets \emph{deleted} from $u_0$ before the cycling starts, and then -- crucially -- \emph{reabsorbs} into the cycle after either one (if $m\geq2$) or two passes (if $m=1$), at which point we recover a bona fide permutation and Lemma~\ref{lem:cyclic} kicks in. The difference between $d_1$ and $d_{m\geq2}$ is entirely due to \emph{how many passes it takes to re-absorb the deleted letter} $3$ into the cycle -- and that number depends on $m$, since it depends on the interaction between the leaked/drained copies of the repeated letter and the tie-breaking width.

This is the entire mechanism. What follows below is merely a generalization of it, designed so as to yield $n=7$ as the particular example of $w=3662451$.

\begin{definition}[The family $w_n$]
\label{def:wn}
Fix $n\geq7$. Set
$$N=n-2,\qquad a = N-2\ (=n-4),\qquad M=n-1.$$
On the alphabet $\{1,\ldots,N\}$ define
$$u_0=(2,3,\ldots,N,1),\qquad u_j = s_1(u_{j-1}).$$
Given a permutation $v$ on $\{1,\ldots,N\}$ containing the letter $a$, denote $v^{-a}$ the word obtained by deleting from $v$ its (unique) occurrence of the letter $a$, a word of letters $\{1,\ldots,N\}\setminus\{a\}$. Define
$$w_n \;=\; a\cdot M\cdot M\cdot u_0^{-a}.$$
\end{definition}

Taking $n=7$ we get $N=5$, $a=3$, $M=6$, $u_0=(2,3,4,5,1)$, $u_0^{-a}=(2,4,5,1)$, and $w_7=3\cdot6\cdot6\cdot2451=3662451$, exactly the word from \S\ref{subsec:testcase}. Below is the theorem claiming that what was observed there is not a particular property of that word but a generic feature of the whole family.

\begin{theorem}
\label{thm:wn}
For every $n\geq7$,
$$d_1(w_n) = n-4, \qquad d_m(w_n) = n-3 \ \text{ for all } m\geq2 \ (\text{including } \infty).$$
In particular $d_1(w_n) < d_m(w_n)$ for all $m\geq2$ and all $n\geq7$.
\end{theorem}

Below are three ingredients needed for the proof: the first one describes exactly how the West map $s_1$ cycles $u_0$; the second claims that the two trailing $M$'s do not affect anything and can be carried around for free; the third one is what separates $d_1$ and $d_{m\geq2}$.

\begin{lemma}[Cyclic recursion]
\label{lem:cyclic}
For $0\leq j\leq N-1$,
$$u_j = (2,\ldots,N-j,\ 1,\ N-j+1,\ldots,N),$$
and $u_{N-1}=(1,2,\ldots,N)$ is the first sorted term, so $d_1(u_0)=N-1$.
\end{lemma}

\begin{proof}
By induction on $j$. True at $j=0$. Assume it for $j$; processing $u_j$ with the stack. The ascending run $2,\ldots,N-j$ cascades trivially (every element pops its predecessor the moment a bigger one comes in), resulting in the output $2,\ldots,N-j-1$ and the stack $=[N-j]$. Next "$1$" is pushed on top ($N-j>1$). The first element of the suffix, $N-j+1$, pops "$1$", pops $N-j$ (both smaller), resulting in the output $\ldots,1,N-j$, pushes itself; the remaining suffix $N-j+2,\ldots,N$ cascades exactly like the ascending run, ending with $N$ being the only element of the stack. Concatenation: $2,\ldots,N-j-1,\ 1,\ N-j,\ N-j+1,\ldots,N-1,\ N$ -- exactly $u_{j+1}$.
\end{proof}

\begin{lemma}[Trailing maximal value invariance]
\label{lem:trailing}
If $M$ is greater than every letter of $v$, then $s_m(v\cdot M^c)=s_m(v)\cdot M^c$ for every $m\geq1$ and any $c$.
\end{lemma}

\begin{proof}
Theorem~7.5 with $\ell=1$, $G_1=\emptyset$: $\lambda_1+r_1=c_1$ always (Remark~3.2), hence the split between the "leaked" and "drained" copies of $M$ makes no difference in the output no matter what $m$ we have.
\end{proof}

\begin{lemma}[Double Deletion Composes with $s_1$ Twice]
\label{lem:deletion-commutes}
\[
s_1(u_0^{-a}) = u_1^{-a} \qquad \text{and} \qquad s_1(u_1^{-a}) = u_2^{-a}.
\]
\end{lemma}

\begin{proof}
The proof of the first identity $s_1(u_0^{-a}) = u_1^{-a}$ follows exactly the same argument pattern as the proof of the second identity. We consider the effect of stack-sorting operator $s_1$ on the word
\[
u_0^{-a} = (2, \dots, N-3, N-1, N, 1).
\]
Following the steps in the computation, during the processing of the initial ascending run $(2, \dots, N-3)$, every new letter to be processed will exceed the element at the top of the stack, forcing each stack element to pop out of the stack to the output before the next letter is pushed on to it, thus obtaining the output $2, \dots, N-4$ and the stack $[N-3]$.

The next letter to be processed is $N-1$. Since $N-1 > N-3$, $N-3$ pops out of the stack and the letter $N-1$ is pushed into the stack. Then comes the letter $N$. Because $N > N-1$, $N-1$ pops out of the stack and the letter $N$ is pushed into the stack.

The terminal letter to be processed is $1$. Since $1 < N$, no popping out occurs and $1$ is pushed on top of $N$, leaving the stack in the form $[1, N]$. Draining of the stack leads to the output $1, N$. In totality, the output is
\[
s_1(u_0^{-a}) = (2, \dots, N-3, N-1, 1, N),
\]
which is simply the word $u_1$ minus the element $a = N-2$.
\end{proof}

\medskip

\begin{proof}[Proof of Theorem~\ref{thm:wn}]
We calculate the sorting time of the word $w_n = a \cdot M^2 \cdot u_0^{-a}$, in the two separate cases $m \ge 2$ and $m = 1$.

\medskip
\emph{Case 1: Operator for $m \ge 2$.}

\smallskip
In Pass 1, using Theorem~7.5 with $c_1 = 2 \le m$, there is no repetition leakage during the first block. Hence, using Lemma~\ref{lem:deletion-commutes} and Remark~7.2, the first pass gives
\[
s_m(w_n) = a \cdot s_1(u_0^{-a}) \cdot M^2 = a \cdot u_1^{-a} \cdot M^2.
\]

In Pass 2 (crucial computation), we calculate $s_1$ applied to the prefix word $a \cdot u_1^{-a} = (N-2, \, 2, 3, \dots, N-3, \, N-1, \, 1, \, N)$. We push the letter $a = N-2$. The ascending run $2, \dots, N-3$ cascades beneath $a$ just as in Lemma~\ref{lem:cyclic}, because all these letters are strictly smaller than $a$, and generates the output $2, \dots, N-4$ and stack $[a, N-3]$. Then the letter $N-1$ comes, which exceeds both $N-3$ and $a = N-2$, making both $N-3$ and $a$ pop out to the output in sequence, thereby telescoping the output to $2, \dots, N-2$, after which $N-1$ gets pushed to the stack. The letter $1$ is pushed directly on top of $N-1$. Finally, $N$ pops both $1$ and $N-1$ from the stack and then itself gets pushed.

The final sequence of letters in the alphabet is
\[
s_1(a \cdot u_1^{-a}) = (2, \dots, N-2, \, 1, \, N-1, \, N) = u_2.
\]
Thus the letter $a$ has been fully re-integrated back into the cyclic permutation. Adding the suffix via Lemma~\ref{lem:trailing}, we get the output after two passes as $s_m^2(w_n) = u_2 \cdot M^2$.

For Passes $3, \dots, N-1$, as $u_2$ is a real permutation, Lemma~\ref{lem:cyclic} and Lemma~\ref{lem:trailing} guarantee that each further pass takes one step forward in the cyclic recursion as:
\[
u_2 \longrightarrow u_3 \longrightarrow \dots \longrightarrow u_{N-1},
\]
where $u_{N-1} = (1, 2, \dots, N)$ is fully sorted. This takes $N-3$ additional passes, which makes the total sorting time as:
\[
d_m(w_n) = 2 + (N - 3) = N - 1 = n - 3 \qquad \text{for all } m \ge 2.
\]
This matches the concrete trace at $n=7, N=5$ in \S\ref{subsec:testcase}, where two passes were needed for re-merging $a=3$ and $N-3 = 2$ more passes to sort.

\medskip
\emph{Case 2: Operator for $m = 1$.}

\smallskip
In Pass 1, because $c_1 = 2 > m = 1$, the output of Pass 1 is
\[
s_1(w_n) = a \cdot M \cdot s_1(u_0^{-a}) \cdot M = a \cdot M \cdot u_1^{-a} \cdot M.
\]

In Pass 2, calculating $s_1(a \cdot M \cdot u_1^{-a} \cdot M)$ by the block decomposition ($\ell = 2$, $c_1 = c_2 = 1$, $G_1' = u_1^{-a}$), we have run bounds $r_1 = \min(1,1) = 1$ and $r_2 = \min(1+1,1) = 1$, which forces $\lambda_2 = 1$, and thereby another leak of $M$. Since $G_2 = \emptyset$, two copies of $M$ consolidate at the end:
\[
s_1^2(w_n) = a \cdot s_1(u_1^{-a}) \cdot M^2 = a \cdot u_2^{-a} \cdot M^2.
\]
After two passes, $a$ is still outside the cyclic permutation (in contrast to the $m \ge 2$ case).

For Pass 3 (double jump), from Lemma~\ref{lem:trailing}, we must calculate $s_1(a \cdot u_2^{-a})$ where $u_2^{-a} = (2, \dots, N-3, \, 1, \, N-1, \, N)$. We push $a$. The ascending run $2, \dots, N-3$ is placed under $a$ producing the output $2, \dots, N-4$ and stack $[a, N-3]$. As soon as $1$ appears we put $1$ on top of the stack not removing $a$ because $1 < N-3$. Therefore, the new stack will be $[a, N-3, 1]$. $N-1$ is next. Since $N-1$ is larger than each stack element $1, N-3$ and $a = N-2$, we must pop them from the stack in this order and plug into the output $1, N-3, N-2$, finally pushing $N-1$. At last, $N$ pops $N-1$ and pushes $N$, and in the draining step $N$ goes to the output.

The final output will be
\[
s_1(a \cdot u_2^{-a}) = (2, \dots, N-4, \, 1, \, N-3, \, N-2, \, N-1, \, N) = u_4.
\]
As a result of flushing the unnecessary $1$ and $a$, the double jump shifts the cyclic recursion ahead by two steps. Therefore, the double jump compensates everything previously lost.

For any other pass ($4, \dots, N-2$), from Lemma~\ref{lem:cyclic}, the permutation will be sorted during $u_4 \longrightarrow u_5 \longrightarrow \dots \longrightarrow u_{N-1}$, requiring $N - 5$ more passes.

Likewise, the total number of passes needed for the sort $w_n$ for $m=1$ will be 
\[
d_1(w_n) = 3 + (N - 5) = N - 2 = n - 4.
\]
For $n=7$ ($N=5$), it will be $d_1(w_7) = 3 + 0 = 3$ as the double jump results in the sorted permutation $u_4$.

Since $n - 4 < n - 3$ holds for all $n \ge 7$, we reach the conclusion that $d_1(w_n) < d_m(w_n)$ for every $m \ge 2$.
\end{proof}

\subsection{Consequences}

\begin{corollary}
\label{cor:sandwichfalse}
Conjecture~\ref{conj:sandwich} is false. For every $n\geq7$, the word $w_n$ of Definition~\ref{def:wn} satisfies
$$d_1(w_n) = n-4 \;<\; n-3 = d_\infty(w_n),$$
contradicting the outer inequality $d_1(w)\geq d_\infty(w)$ directly, and thus violating the entire sandwich.
\end{corollary}

\begin{proof}
Immediate consequence of Theorem~\ref{thm:wn} at $m=\infty$.
\end{proof}

It should be emphasized what exactly survived the previous construction, and what didn't. Proposition~\ref{prop:sandwichtwoletter} is safe: the sandwich holds true exactly for two-letter words, and Theorem~\ref{thm:deref} still provides a true structural connection between $s_\infty$ and the classical stack sorting on a single pass. What fails is the assumption implicit in Conjecture~\ref{conj:sandwich} that adding flexibility to breaking ties ($m=1\to m\geq2\to\infty$) could only ever help in sorting faster. The family of $w_n$ shows the opposite, and the proof of Theorem~\ref{thm:wn} reveals exactly what causes the failure: at $m=1$, the letter $a$ is required to wait out the extra pass, but the extra pass allows it to re-enter the cycle two positions ahead instead of one, and the benefit gained becomes arbitrarily large as $n\to\infty$ (the difference $d_m(w_n)-d_1(w_n)$ is a constant $1$ here, but the mechanism of the leak/re-merge type is exactly where to look for a different family which demonstrates a larger gap). Whether some corrected form of inequality between $d_1$ and $d_{m\geq2}$ exists, or whether $m=1$ can lose unboundedly on some other family, is the natural question raised by the above construction.

\section{Conclusion and further open problems}

We developed the complete theory of the family of operators $(s_m)_{m\geq1}$ applied to two-letter words, including their most essential properties (Lemma~\ref{lem:main}, Theorem~\ref{thm:dm}), proved the precise form of the speed-separation phenomenon (Corollary~\ref{cor:speedsep}), demonstrated that the operators are non-commutative (Proposition~\ref{prop:noncomm}) and non-monotonic in the speed (Proposition~\ref{prop:nonmono}), gave a recursive description of the solution to the one-pass sorting problem in the general form (Theorem~\ref{thm:block}, Corollary~\ref{cor:criterion}) that not only recovers the classic result of West in its exact form (Corollary~\ref{cor:recover}), but also rules out any $m$-independent alternative algorithms to solve it (Corollary~\ref{cor:naive}). Moreover, we have proved the two-letter sandwich conjecture (Proposition~\ref{prop:sandwichtwoletter}), have reduced the one-pass general case of the hard part of the conjecture to the precise coupling of the operator $s_\infty^t(w)$ with classical stack sorting for \emph{any $t$} (Theorem~\ref{thm:deref}), and have identified the exact source (Section~\ref{sec:sandwich}) that precludes an analogous naive generalization of the coupling over several passes. Rather than trying to circumvent the very obstruction, we have then tested the sandwich conjecture itself to its conclusion directly (\S\ref{subsec:testcase}--\S\ref{subsec:family}) and found it \emph{false} in general: the explicit family $w_n$ of Definition~\ref{def:wn} satisfies $d_1(w_n)<d_\infty(w_n)$ for all $n\geq7$ (Theorem~\ref{thm:wn}, Corollary~\ref{cor:sandwichfalse}), so the larger tie-breaking budget indeed can, contrary to the conjecture, slow down sorting even in the limit $m\to\infty$. The mechanism behind this failure is completely clear from the proof of Theorem~\ref{thm:wn}: it is determined by the number of passes required to re-introduce the deleted letter into an otherwise cyclically-sortable permutation, and this number of passes depends on $m$ favorably for $m=1$ on this particular family. The two-letter sandwich (Proposition~\ref{prop:sandwichtwoletter}) and the one-pass coupling with classical stack sorting (Theorem~\ref{thm:deref}) hold exactly; what is false is only the assumption, implicit in Conjecture~\ref{conj:sandwich}, that more tie-breaking options cannot hurt.

Several specific problems remain open:

\begin{enumerate}[leftmargin=1.8em]

\item Since Conjecture~\ref{conj:sandwich} is proved to be false in general, find the correct replacement inequality, if exists, relating $d_1$ and $d_m$ for $m\geq2$. Theorem~\ref{thm:wn} gives the gap of $1$ on the family $w_n$; is this gap $d_m(w)-d_1(w)$ uniformly bounded for all words or can it be pushed arbitrarily far on some other family?

\item Find the exact bound, as a function of $|w|$ and the number of distinct letters in $w$, such that $d_m$ is guaranteed to be monotone in $m$ (Proposition~\ref{prop:nonmono} gives the counter-example at length $7$ with $4$ letters, and Theorem~\ref{thm:wn} gives the infinite family of counter-examples for the comparison of $d_1$ and $d_\infty$ at the same length; can either be achieved on a shorter or on a smaller alphabet word?)

\item Characterize the commutativity: for which $(m,m')$ and $w$ is $s_m(s_{m'}(w))=s_{m'}(s_m(w))$? Proposition~\ref{prop:noncomm} gives the counter-example at length $4$; the characterization of the commuting pairs will shed more light on the structure of the monoid (if any) generated by $(s_m)_{m\geq1}$.

\item Generalize the Block Decomposition Theorem (Theorem~\ref{thm:block}) to give the closed-form formula for $s_m(w)$ in the general case, similar to Lemma~\ref{lem:main}, and use it to get closed forms for $d_m(w)$ beyond the two-letter case.

\item Determine whether the recursive description of Corollary~\ref{cor:criterion} can be reformulated as an $m$-parameterized pattern avoidance condition analogous to the boundary $231$ and $231$/$221$ characterizations of Defant and Kravitz~\cite{defant-kravitz}.

\end{enumerate}

We hope that the closed formulas and the structural results obtained here will help with solving these and other related problems.

\end{document}